\documentclass{amsart}
\usepackage{amsmath}
\usepackage{amssymb}
\usepackage{amsthm}
\usepackage[msc-links,abbrev]{amsrefs}
\usepackage{hyperref}
\usepackage[noabbrev,capitalize]{cleveref}
\usepackage{mathrsfs}
\usepackage{mathtools}
\usepackage{slashed}
\usepackage{tikz-cd}
\usepackage{enumitem}

\setenumerate{label=(\roman*)}

\DeclareMathOperator{\tr}{tr}

\DeclareMathOperator{\dvol}{dV}

\DeclareMathOperator{\darea}{dA}

\DeclareMathOperator{\Ric}{Ric}

\DeclareMathOperator{\Rm}{Rm}

\DeclareMathOperator{\logpart}{lp}

\newcommand{\defn}[1]{{\boldmath\bfseries#1}}

\newcommand{\cg}{\widetilde{g}}

\newcommand{\cf}{\widetilde{f}}

\newcommand{\cu}{\widetilde{u}}
\newcommand{\cv}{\widetilde{v}}
\newcommand{\cw}{\widetilde{w}}

\newcommand{\cD}{\widetilde{D}}

\newcommand{\cI}{\widetilde{I}}

\newcommand{\cR}{\widetilde{R}}

\newcommand{\cT}{\widetilde{T}}

\newcommand{\cnabla}{\widetilde{\nabla}}

\newcommand{\cDelta}{\widetilde{\Delta}}
\newcommand{\cGamma}{\widetilde{\Gamma}}

\newcommand{\cmE}{\widetilde{\mathcal{E}}}

\newcommand{\cmG}{\widetilde{\mathcal{G}}}

\DeclareMathOperator{\cRm}{\widetilde{\Rm}}

\newcommand{\hf}{\widehat{f}}

\newcommand{\hu}{\widehat{u}}
\newcommand{\hv}{\widehat{v}}

\newcommand{\hvarphi}{\widehat{\varphi}}
\newcommand{\hpsi}{\widehat{\psi}}

\newcommand{\lv}{\lvert}
\newcommand{\rv}{\rvert}
\newcommand{\lV}{\lVert}
\newcommand{\rV}{\rVert}

\newcommand{\mE}{\mathcal{E}}

\newcommand{\mG}{\mathcal{G}}

\newcommand{\mI}{\mathcal{I}}

\newcommand{\mL}{\mathcal{L}}

\newcommand{\mU}{\mathcal{U}}

\newcommand{\kD}{\mathfrak{D}}

\newcommand{\kc}{\mathfrak{c}}

\newcommand{\bN}{\mathbb{N}}

\newcommand{\bR}{\mathbb{R}}

\newcommand{\onf}{\mathsf{n}}

\def\sideremark#1{\ifvmode\leavevmode\fi\vadjust{\vbox to0pt{\vss
 \hbox to 0pt{\hskip\hsize\hskip1em
 \vbox{\hsize3cm\tiny\raggedright\pretolerance10000
 \noindent #1\hfill}\hss}\vbox to8pt{\vfil}\vss}}}

\newcommand{\suchthat}{\mathrel{}:\mathrel{}}

\newtheorem{theorem}{Theorem}[section]
\newtheorem{proposition}[theorem]{Proposition}
\newtheorem{lemma}[theorem]{Lemma}
\newtheorem{corollary}[theorem]{Corollary}

\theoremstyle{definition}

\theoremstyle{remark}

\numberwithin{equation}{section}

\makeatletter
\@namedef{subjclassname@2020}{\textup{2020} Mathematics Subject Classification}
\makeatother

\begin{document}

\title[Formal self-adjointness of curved Ovsienko--Redou operators]{Formal self-adjointness of a family of conformally invariant bidifferential operators}
\author{Jeffrey S.\ Case}
\address{Department of Mathematics \\ Penn State University \\ University Park, PA 16802 \\ USA}
\email{jscase@psu.edu}
\author{Zetian Yan}
\address{Department of Mathematics \\ UC Santa Barbara \\ Santa Barbara \\ CA 93106 \\ USA}
\email{ztyan@ucsb.edu}
\begin{abstract}
 We prove that the curved Ovsienko--Redou operators and a related family of differential operators are formally self-adjoint.
 This verifies two conjectures of Case, Lin, and Yuan.
\end{abstract}
\maketitle

\section{Introduction}
\label{sec:intro}

The GJMS operator of order $2k$, introduced by Graham, Jenne, Mason, and Sparling~\cite{GJMS1992}, is a conformally invariant differential operator with leading-order term $\Delta^k$ defined on any Riemannian manifold $(M^n,g)$ of dimension $n \geq 2k$.
It can be defined by suitably restricting the power $\cDelta^k$ of the Laplacian of the Fefferman--Graham ambient space~\cite{FeffermanGraham2012}.
While $\cDelta^k$ is formally self-adjoint in the ambient space, this does not immediately imply that the GJMS operators are formally self-adjoint.

There are now two proofs that the GJMS operators are formally self-adjoint.
The first, due to Graham and Zworski~\cite{GrahamZworski2003} (cf.\ \cite{FeffermanGraham2002}), uses a realization of the GJMS operators as poles of the scattering operator of a Poincar\'e space.
The second, due to Juhl~\cite{Juhl2013} (cf.\ \cite{FeffermanGraham2013}), uses a remarkable formula for the GJMS operators in terms of a family of formally self-adjoint second-order differential operators.
The first approach relies heavily on the equivalent characterization of the GJMS operators as obstructions to formally extending a function to be harmonic in the ambient space, while the second approach requires a complicated combinatorial argument.

Case, Lin, and Yuan~\cite{CaseLinYuan2022or} gave two generalizations of the GJMS operators;
here we discuss those which are plausibly formally self-adjoint.
First is a family of conformally invariant bidifferential operators
\begin{equation*}
 D_{2k} \colon \mE\left[ -\frac{n-2k}{3} \right]^{\otimes 2} \to \mE\left[-\frac{2n+2k}{3}\right]
\end{equation*}
of total order $2k$.
They are called the \defn{curved Ovsienko--Redou operators} bcause they generalize a family of bidifferential operators constructed by Ovsienko and Redou~\cite{OvsienkoRedou2003} on the sphere.
The operators $D_{2k}$ are determined ambiently by
\begin{align*}
 \cD_{2k}(\cu \otimes \cv) & := \sum_{r+s+t=k} a_{r,s,t} \cDelta^r\left( (\cDelta^s\cu)(\cDelta^t\cv) \right) , \\
 a_{r,s,t} & := \frac{k!}{r!s!t!}\frac{\Gamma\bigl(\frac{n+4k}{6}-r\bigr)\Gamma\bigl(\frac{n+4k}{6}-s\bigr)\Gamma\bigl(\frac{n+4k}{6}-t\bigr)}{\Gamma\bigl(\frac{n-2k}{6}\bigr)\Gamma\bigl(\frac{n+4k}{6}\bigr)^2} ,
\end{align*}
on $\cmE\bigl[-\frac{n-2k}{3}\bigr] \otimes \cmE\bigl[-\frac{n-2k}{3}\bigr]$;
in this paper, tensor products are over $\bR$.
Second is a family of conformally invariant differential operators
\begin{equation*}
 D_{2k,\mI} \colon \mE\left[ -\frac{n-2k-2\ell}{2} \right] \to \mE\left[-\frac{n+2k+2\ell}{2} \right]
\end{equation*}
of order $2k$ associated to a scalar Weyl invariant $\mI$ of weight $-2\ell$.
These are determined ambiently by a scalar Riemannian invariant $\cI$ of weight $-2\ell$ and
\begin{align*}
 \cD_{2k,\cI}(\cu) & := \sum_{r+s=k} b_{r,s}\cDelta^r\left( \cI \cDelta^s\cu \right) , \\
 b_{r,s} & := \frac{k!}{r!s!}\frac{(\ell+s-1)!(\ell+r-1)!}{(\ell-1)!^2} ,
\end{align*}
on $\cmE\bigl[-\frac{n-2k-2\ell}{2}\bigr]$.
See \cref{sec:bg} for an explanation of our notation and a description of how the ambient formulas determine conformally invariant operators.

The symmetry of the coefficients $a_{r,s,t}$ and $b_{r,s}$ and the observations
\begin{align*}
 \cmE\left[-\frac{n-2k}{3}\right] \ni \cu, \cv, \cw & \Longrightarrow \cu\cD_{2k}(\cv \otimes \cw) \in \cmE[-n] , \\
 \cmE\left[-\frac{n-2k-2\ell}{2}\right] \ni \cu, \cv & \Longrightarrow \cu\cD_{2k,\cI}(\cv) \in \cmE[-n] ,
\end{align*}
imply that the induced operators $D_{2k}$ and $D_{2k,\mI}$ are plausibly formally self-adjoint.
Case, Lin, and Yuan~\cite{CaseLinYuan2022or} conjectured their formally self-adjointness and verified this when $k \leq 3$.
The difficulty in this problem is that there is not an equivalent description of these operators as an obstruction to solving some second-order PDE, and hence the Graham--Zworski argument cannot be adapted to $D_{2k}$ or $D_{2k,\mI}$.
Case, Lin, and Yuan instead adapted Juhl's approach, with the constraint $k \leq 3$ due to the difficulty of the combinatorial argument.

In this paper we develop a new, conceptually simple approach to proving formal self-adjointness which verifies the Case--Lin--Yuan conjecture:

\begin{theorem}
 \label{ovsienko-redou-are-fsa}
 Let $(M^n,g)$ be a Riemannian manifold and let $k \in \bN$;
 if $n$ is even, then assume additionally that $k \leq n/2$.
 Then $D_{2k} \colon \mE\bigl[-\frac{n-2k}{3}\bigr]^{\otimes 2} \to \mE\bigl[-\frac{2n+2k}{3}\bigr]$ is formally self-adjoint.
\end{theorem}

\begin{theorem}
 \label{linear-are-fsa}
 Let $(M^n,g)$ be a Riemannian manifold, let $k \in \bN$, and let $\mI$ be a scalar Weyl invariant of weight $-2\ell$ in the ambient space;
 if $n$ is even, then assume additionally that $k \leq n/2$ if $\ell=0$ and $k+\ell \leq n/2+1$ if $\ell\geq 1$.
 Then $D_{2k,\mI} \colon \mE\bigl[-\frac{n-2k-2\ell}{2}\bigr] \to \mE\bigl[-\frac{n+2k+2\ell}{2}\bigr]$ is formally self-adjoint.
\end{theorem}

The restrictions on $k$ and $k+\ell$ when $n$ is even ensure that the operators $D_{2k}$ and $D_{2k,\mI}$, respectively, are independent of the ambiguity of the ambient metric;
see \cref{sec:bg} for details.
When $\mI=1$, the operator $D_{2k,\mI}$ is a nonzero multiple of the GJMS operator of order $2k$, and so \cref{linear-are-fsa} gives a new proof of its formal self-adjointness.

The main idea in the proofs of \cref{ovsienko-redou-are-fsa,linear-are-fsa} is to restrict the ambient operators $\cD_{2k}$ and $\cD_{2k,\cI}$ to the Poincar\'e space $(\mathring{X},g_+)$ equivalent to a given ambient space.
Their Dirichlet forms
\begin{align*}
 \int_{r > \varepsilon} \left.\left( \cu \cD_{2k}( \cv \otimes \cw) \right)\right|_{\mathring{X}} \dvol_{g_+} \quad\text{and}\quad \int_{r > \varepsilon} \left.\left( \cu \cD_{2k,\cI}(\cv) \right)\right|_{\mathring{X}} \dvol_{g_+}
\end{align*}
are both in $O(\log\varepsilon)$, and the coefficient of the log terms are the Dirichlet forms of the respective operators $D_{2k}$ and $D_{2k,\mI}$ on $M$.
The Divergence Theorem implies that there are polydifferential operators $B$ and $B_{\cI}$ such that
\begin{align*}
 \int_{r>\varepsilon} \left. \left( \cu D_{2k}( \cv \otimes \cw) - \cv \cD_{2k}( \cu \otimes \cw) \right) \right|_{\mathring{X}} \dvol_{g_+} & = \int_{r=\varepsilon} B(u \otimes v \otimes w) \darea_{g_+} , \\
 \int_{r>\varepsilon} \left.\left( \cu \cD_{2k,\cI}(\cv) - \cv \cD_{2k,\cI}(\cu) \right) \right|_{\mathring{X}} \dvol_{g_+} & = \int_{r=\varepsilon} B_{\cI}(u \otimes v) \darea_{g_+} .
\end{align*}
Moreover, the right-hand sides cannot have a log term---they are in $O(1)$---and hence $D_{2k}$ and $D_{2k,\mI}$ are formally self-adjoint;
see \cref{sec:conclusion} for details.

This paper is organized as follows:

In \cref{sec:bg} we recall necessary background material on the ambient space and give a unified construction of the operators $D_{2k}$ and $D_{2k,\mI}$ of \cref{ovsienko-redou-are-fsa} and \cref{linear-are-fsa}, respectively.
Our construction differs slightly from that of Case, Lin, and Yuan.
We normalize our operators so that $D_{2k}$ coincides with their construction of the curved Ovsienko--Redou operators.
While the operators $D_{2k,\mI}$ do not coincide with the linear operators constructed by Case, Lin, and Yuan~\cite{CaseLinYuan2022or}*{Theorem~1.7}, the linear spans of our respective operators agree.

In \cref{sec:scattering} we describe the restrictions of $\cD_{2k}$ and $\cD_{2k,\cI}$ to $(\mathring{X},g_+)$ and express them in terms of linear combinations of compositions of $\Delta_{g_+} + c$ for constants $c \in \bR$.

In \cref{sec:conclusion} we realize the Dirichlet forms for $D_{2k}$ and $D_{2k,\mI}$ in terms of the leading-order term of the corresponding Dirichlet forms in $(\mathring{X},g_+)$.
We then use the simple formulas for the latter operators to deduce \cref{ovsienko-redou-are-fsa,linear-are-fsa}.

\section{Ambient spaces and the Ovsienko--Redou operators}
\label{sec:bg}

In this section we recall the construction~\cite{CaseLinYuan2022or} of the curved Ovsienko--Redou operators and their related linear operators.

\subsection{Ambient spaces}
\label{subsec:ambient}
We begin by recalling the relevant aspects of the ambient space, following Fefferman and Graham~\cite{FeffermanGraham2012}.

Let $(M^n,\kc)$ be a conformal manifold of signature $(p,q)$.
Denote
\begin{equation*}
 \mG := \left\{ (x,g_x) \suchthat x \in M, g \in \kc \right\} \subset S^2T^\ast M 
\end{equation*}
and let $\pi \colon \mG \to M$ be the natural projection.
We regard $\mG$ as a principal $\bR_+$-bundle with dilation $\delta_\lambda \colon \mG \to \mG$, $\lambda \in \bR_+$, given by
\begin{equation*}
 \delta_\lambda (x, g_x) := (x , \lambda^2 g_x) .
\end{equation*}
Denote by $T := \left. \frac{\partial}{\partial\lambda}\right|_{\lambda=1} \delta_\lambda$ the infinitesimal generator of $\delta_\lambda$.
The canonical metric is the degenerate metric ${\boldsymbol{g}}$ on $\mG$ defined by
\begin{equation*}
    {\boldsymbol{g}}(X,Y) := g_x(\pi_\ast X, \pi_\ast Y)
\end{equation*}
for $X,Y\in T_{(x,g_x)}\mG$.
Note that $\delta_\lambda^\ast \boldsymbol{g} = \lambda^2\boldsymbol{g}$.

A choice of representative $g \in \kc$ determines an identification $\bR_+ \times M \cong \mG$ via $(t,x)\cong (x,t^2g_x)$.
In these coordinates, $T_{(t,x)} = t\partial_t$ and $\boldsymbol{g}_{(t,x)} = t^2\pi^\ast g$.

Extend the projection and dilation to $\mG \times \bR$ in the natural way:
\begin{align*}
 \pi( x, g_x, \rho) & := x , \\
 \delta_\lambda( x, g_x, \rho) & := ( x , \lambda^2 g_x , \rho ) ,
\end{align*}
where $\rho$ denotes the coordinate on $\bR$.
We abuse notation and also denote by $T$ the infinitesimal generator of $\delta_\lambda \colon \mG \times \bR \to \mG \times \bR$.
Let $\iota \colon \mG \to \mG \times \bR$ denote the inclusion $\iota(x,g_x) := (x,g_x,0)$.
A \defn{pre-ambient space} for $(M^n,\kc)$ is a pair $(\cmG,\cg)$ consisting of a dilation-invariant subspace $\cmG \subseteq \mG \times \bR$ containing $\iota(\mG)$ and a pseudo-Riemannian metric $\cg$ of signature $(p+1,q+1)$ satisfying $\delta_\lambda^\ast \cg = \lambda^2 \cg$ and $\iota^\ast \cg = \boldsymbol{g}$.

An \defn{ambient space} for $(M^n,\kc)$ is a pre-ambient space $(\cmG,\cg)$ for $(M^n,\kc)$ which is formally Ricci flat;
i.e.
\begin{itemize}
 \item if $n$ is odd, then $\Ric(\cg) \in O(\rho^\infty)$;
 \item if $n$ is even, then $\Ric(\cg) \in O^+(\rho^{n/2-1})$.
\end{itemize}
Here $O^+(\rho^m)$ is the set of sections $S$ of $S^2T^\ast\cmG$ such that
\begin{enumerate}
 \item $\rho^{-m}S$ extends continuously to $\iota(\mG)$, and
 \item for each $z = (x,g_x) \in \mG$, there is an $s \in S^2 T_{x}^\ast M$ such that $\tr_{g_x}s=0$ and $(\iota^\ast(\rho^{-m}S)(z) = (\pi^\ast s)(z)$.
\end{enumerate}
Two ambient spaces $(\cmG_j,\cg_j)$, $j \in \{ 1, 2 \}$, for $(M^n,\kc)$ are \defn{ambient-equivalent} if there are open sets $\mU_j \subseteq \cmG_j$ and a diffeomorphism $\Phi \colon \mU_1 \to \mU_2$ such that
\begin{enumerate}
 \item $\iota(\mG) \subseteq \mU_j$ for each $j \in \{ 1, 2 \}$;
 \item each $\mU_j$ is dilation-invariant and $\Phi$ commutes with dilations;
 \item $\Phi \circ \iota = \iota$;
 \item if $n$ is odd, then $\Phi^\ast\cg_2 - \cg_1 \in O(\rho^\infty)$;
 \item if $n$ is even, then $\Phi^\ast\cg_2 - \cg_1 \in O^+(\rho^{n/2})$.
\end{enumerate}
Fefferman and Graham showed~\cite{FeffermanGraham2012}*{Theorem~2.3} is that if $(M^n,\kc)$ is a conformal manifold, then there is a unique, up to ambient-equivalence, ambient space for $(M^n,\kc)$.
In fact, they proved a stronger statement~\cite{FeffermanGraham2012}*{Theorem~2.9(A)}:
Let $(M^n,\kc)$ be a conformal manifold and pick a representative $g \in \kc$.
Then there is an $\varepsilon>0$ and a one-parameter family $g_\rho$, $\rho \in (-\varepsilon,\varepsilon)$, of metrics on $M$ such that $g_0=g$ and
\begin{equation}
 \label{eqn:straight-and-normal}
 \begin{split}
  \cmG & := \mG \times (-\varepsilon,\varepsilon) , \\
  \cg & := 2\rho \, dt^2 + 2t \, dt \, d\rho + t^2 g_\rho 
 \end{split}
\end{equation}
defines an ambient space $(\cmG,\cg)$ for $(M^n,\kc)$.
We say that an ambient metric in the form above is \defn{straight and normal}.

Let $(\cmG,\cg)$ be the ambient space for $(M^n,\kc)$.
Denote by
\begin{equation*}
 \cmE[w] := \left\{ \cf \in C^\infty(\cmG) \suchthat \delta_\lambda^\ast \cf = \lambda^w\cf \right\} 
\end{equation*}
the space of homogeneous functions on $\cmG$ of weight $w \in \bR$.
Note that $\cf \in \cmE[w]$ if and only if $T\cf = w\cf$.
The space of \defn{conformal densities} of weight $w$ is
\begin{equation*}
 \mE[w] := \left\{ \iota^\ast \cf \in C^\infty(\mG) \suchthat \cf \in \cmE[w] \right\} .
\end{equation*}

Fix $n \in \bN$.
An ambient \defn{scalar Riemannian invariant} $\cI$ is an assignment to each ambient space $(\cmG^{n+2},\cg)$ of a linear combination $\cI_{\cg}$ of complete contractions of
\begin{equation}
 \label{eqn:tensors}
 \cnabla^{N_1}\cRm \otimes \dotsm \otimes \cnabla^{N_\ell}\cRm , 
\end{equation}
with $\ell \geq 2$, where $\cnabla$ and $\cRm$ are the Levi-Civita connection and Riemann curvature tensor, respectively, of $\cg$, we regard $\cRm$ as a section of $\otimes^4T^\ast\cmG$, and we use $\cg^{-1}$ to take contractions.
Any complete contraction of~\eqref{eqn:tensors} is homogeneous of weight
\begin{equation*}
 w = -2\ell - \sum_{i=1}^\ell N_i .
\end{equation*}
We assume $\ell \geq 2$ because any complete contraction of $\cnabla^N\cRm$ is proportional to $\cDelta^{N/2}\cR$ modulo ambient scalar Riemannian invariants, and $\cDelta^{N/2}\cR=0$ when it is independent of the ambiguity of $\cg$.
If $\cI$ is independent of the ambiguity of $\cg$, then $\mI := \iota^\ast \cI_{\cg} \in \mE[w]$ is independent of the choice of ambient space.
A \defn{scalar Weyl invariant} is a scalar invariant $\mI \in \mE[w]$ constructed in this way.
Fefferman and Graham gave a condition on the weight $w$ which implies this independence:

\begin{lemma}[\cite{FeffermanGraham2012}*{Proposition~9.1}]
 \label{fefferman-graham-lemma}
 Let $(\cmG^{n+2},\cg)$ be a straight and normal ambient space and let $\cI \in \cmE[w]$ be an ambient scalar Riemannian invariant.
 If $w \geq -n - 2$, then $\iota^\ast\cI_{\cg}$ is independent of the ambiguity of $\cg$.
\end{lemma}

Bailey, Eastwood, and Graham~\cite{BaileyEastwoodGraham1994}*{Theorem~A} showed that every conformally invariant scalar of weight $w>-n$ is a Weyl invariant.

\subsection{Conformally invariant polydifferential operators}
\label{subsec:ovsienko-redou}

Fix $k,n \in \bN$.
An ambient \defn{polydifferential operator} $\cD$ of weight $-2k$ is an assignment to each ambient space $(\cmG^{n+2},\cg)$ of a linear map
\begin{equation*}
 \cD^{\cg} \colon \cmE[w_1] \otimes \dotsm \otimes \cmE[w_j] \to \cmE[w_1 + \dotsm + w_j - 2k]
\end{equation*}
such that $\cD^{\cg}( \cu_1 \otimes \dotsm \otimes \cu_j)$ is a linear combination of complete contractions of
\begin{equation}
 \label{eqn:polydifferential-precontraction}
 \cnabla^{N_1}\cu_1 \otimes \dotsm \otimes \cnabla^{N_j}\cu_j \otimes \cnabla^{N_{j+1}}\cRm \otimes \dotsm \otimes \cnabla^{N_\ell}\cRm
\end{equation}
with $\ell=j$ or $\ell \geq j+2$.
Necessarily the powers $N_1,\dotsc,N_\ell$ satisfy
\begin{equation*}
 \sum_{i=1}^\ell N_i + 2\ell - 2j = 2k.
\end{equation*}
The \defn{total order} of such a contraction is $\sum_{i=1}^j N_i$.
We say that $\cD$ is \defn{tangential} if $\iota^\ast(\cD^{\cg}(\cu_1 \otimes \dotsm \otimes \cu_j))$ depends only on $\iota^\ast\cu_1,\dotsc,\iota^\ast\cu_j$ and $\cg$ modulo its ambiguity.
On each conformal manifold $(M^n,\kc)$, such an operator determines a \defn{conformally invariant polydifferential operator} 
\begin{equation*}
 D \colon \mE[w_1] \otimes \dotsm \otimes \mE[w_j] \to \mE[w_1 + \dotsm + w_j - 2k].
\end{equation*}

We will give a condition on the total order of an ambient polydifferential operator which implies that it is independent of the ambiguity of $\cg$.
Let $(\cmG,\cg)$ be a straight and normal ambient space for $(M^n,g)$.
Following Fefferman and Graham~\cite{FeffermanGraham2012}, given coordinates $\{ x^i \}_{i=1}^n$ on $M$ and a multi-index $A = (a_1,\dotsc,a_r) \in \{ 0, 1, \dotsc, n, \infty \}^r$, we denote by $\cT_A = \cT_{a_1 \dotsm a_r}$ a component of a tensor $\cT \in \otimes^kT^\ast\cmG$, where the index $0$ represents $\partial_t$, an index $i \in \{ 1, \dotsc, n \}$ represents $\partial_{x^i}$, and the index $\infty$ represents $\partial_\rho$.
The \defn{length} of $A$ is $\lv A \rv := r$ and the \defn{strength} of $A$ is
\begin{equation*}
 \lV A \rV := \left|\left\{ j \suchthat a_j \in \{ 1, \dotsc, n \} \right\}\right| + 2\left|\left\{ j \suchthat a_j = \infty \right\}\right| .
\end{equation*}
When $\cf \in C^\infty(\cmG)$, we denote $\cf_{a_1 \dotsm a_r} := \cnabla_{a_r} \dotsm \cnabla_{a_1} \cf =: (\cnabla^r\cf)_{a_r \dotsm a_1}$.

We first determine the dependence of the covariant derivatives $\cnabla^r\cf$ of a function $\cf \in C^\infty(\cmG)$ on the ambiguity of $\cg$ (cf.\ \cite{FeffermanGraham2012}*{Proof of Proposition~6.2}).

\begin{lemma}
 \label{f-invariant}
 Let $(\cmG^{n+2},\cg)$ be a straight and normal ambient space and let $\cf \in \cmE[w]$.
 Let $A$ be a multi-index of length $r := \lv A \rv$ and strength $s := \lV A \rV$.
 Then $\cf_A$ mod $O(\rho^{(n-s)/2})$ depends only on $\cg$ mod $O(\rho^{n/2})$.
\end{lemma}

\begin{proof}
 The proof is by induction in $r$.
 Since $\cf \in \cmE[w]$, we see that $\mL_T\cnabla^r\cf = w\cnabla^r\cf$.
 Combining this with the identity $\cnabla T = \cg$ yields
 \begin{equation}
  \label{eqn:remove-T-from-f}
  \cnabla_T\cnabla^r\cf = (w-r)\cnabla^r\cf .
 \end{equation}
 In particular, it suffices to consider the case $(n-s)/2 \leq n/2-1$:
 If $A$ is a multi-index of strength $s \leq 1$, then Equation~\eqref{eqn:remove-T-from-f} implies that $\cf_A$ is proportional to $\cf$ (if $s=0$) or $\cf_a$ for some $a \in \{ 1, \dotsc, n \}$ (if $s=1$).
 In either case $\cf_A$ is independent of $\cg$.
 
 We now proceed with the induction.
 The claim is trivially true when $r=0$.
 For the inductive step, consider $\cf_{Aa}$ for $A=(a_1,\dotsc,a_r)$ a multi-index of length $r$ and $Aa := (a_1, \dotsc, a_r, a)$ the corresponding multi-index of length $r+1$.
 Denote $s := \lV Aa \rV$ and $s' := \lV A \rV$.
 
 Suppose first that $a=0$.
 Then Equation~\eqref{eqn:remove-T-from-f} yields $\cf_{Aa} = (w-r)\cf_A$.
 The required dependence of $\cf_{Aa}$ then follows from the inductive hypothesis.
 
 Suppose next that $a \not= 0$.
 Write
 \begin{equation*}
  \cf_{Aa} = \partial_a\cf_A - \cGamma_{aa_1}^b\cf_{ba_2\dotsm a_r} - \dotsm - \cGamma_{aa_r}^b\cf_{a_1\dotsm a_{r-1}b} .
 \end{equation*}
 If $a \not= \infty$, then $s = s'+1$ and hence $\partial_{x^i}O(\rho^{(n-s')/2}) \subseteq O(\rho^{(n-s)/2})$.
 If instead $a = \infty$, then $s = s'+2$ and hence $\partial_\rho O(\rho^{(n-s')/2}) = O(\rho^{(n-s)/2})$.
 In either case, the inductive hypothesis then gives the required dependence of $\partial_a\cf_A$.
 We now consider the summand $\cGamma_{aa_r}^b\cf_{a_1\dotsm a_{r-1}b}$;
 the remaining summands are similar.
 Since $(n-s)/2 \leq n/2-1$, the formula~\cite{FeffermanGraham2012}*{Equation~(3.16)} for the Christoffel symbols of $\cg$ imply that $\cGamma_{IJ}^K$ mod $O(\rho^{(n-s)/2})$ depends only on $\cg$ mod $O(\rho^{n/2})$, and so we can ignore the ambiguity caused by the Christoffel symbols.
 The worst case for the ambiguity caused by the inductive hypothesis is when $b=\infty$ and $a_r=0$, so that $\lV (a_1, \dotsc, a_{r-1}, b) \rV = s'+2$.
 If $a = \infty$, then $s'+2=s$, and so the required dependence follows from the inductive hypothesis.
 If $a \not= \infty$, then the vanishing of the Christoffel symbol $\Gamma_{a0}^\infty$ implies that the worst case does not happen.
 Thus $\lV (a_1, \dotsc, a_{r-1}, b) \rV \leq s'+1 = s$, and so again the required dependence follows from the inductive hypothesis.
\end{proof}

We now give a condition on the total order of an ambient polydifferential operator that implies that it is independent of the ambiguity of $\cg$ (cf.\ \cite{FeffermanGraham2012}*{Proposition~9.1}).

\begin{corollary}
 \label{dependence}
 Let $(\cmG^{n+2},\cg)$ be a straight and normal ambient space and let $\cD$ be an ambient polydifferential operator of weight $-2k$.
 Suppose that
 \begin{enumerate}
  \item $n$ is odd,
  \item $k \leq n/2$, or
  \item $k \leq n/2 + 1$ and $\cD$ can be expressed as a linear combination of complete contractions of tensors of the form~\eqref{eqn:polydifferential-precontraction} with $\ell \geq j + 2$.
 \end{enumerate}
 Then $\cD$ is independent of the ambiguity of $\cg$.
\end{corollary}

\begin{proof}
 The conclusion is clear when $n$ is odd.
 Assume that $n$ is even.
 
 Since the component $\cg^{ab}$ at $\rho=0$ of the inverse ambient metric is nonzero only if $\lV ab \rV = 2$, we see that a complete contraction of a tensor of the form~\eqref{eqn:polydifferential-precontraction} can only be nonzero on $\mG$ if each contracted pair has strength two.
 Denote by $S_1,\dotsc,S_\ell$ the strengths of the factors in a contributing monomial.
 Note that the sum of the strengths must equal the total number of indices in a given monomial.
 Therefore
 \begin{equation}
  \label{eqn:sum-strengths}
  \sum_{i=1}^\ell S_i = \sum_{i=1}^\ell N_i + 4(\ell-j) = 2(k + \ell-j)  .
 \end{equation}
 Clearly $S_i \geq 0$ if $i \leq j$.
 Fefferman and Graham observed~\cite{FeffermanGraham2012}*{Proposition~6.1} that $S_i \geq 4$ if $i \geq j+1$.
 
 Suppose first that $k \leq n/2$.
 If $i_0 \in \{ 1, \dotsc, j \}$, then Equation~\eqref{eqn:sum-strengths} yields
 \begin{equation*}
  S_{i_0} \leq \sum_{i=1}^\ell S_i - 4(\ell-j) = 2(k-\ell+j) \leq 2k \leq n .
 \end{equation*}
 Moreover, equality holds if and only if $k=n/2$ and $\ell=j$ and $S_i=0$ for each $i \not= i_0$.
 If equality holds, then our monomial is proportional to $(\prod_{i\not=i_0}\cu_i)\cDelta^{n/2}\cu_{i_0}$ modulo complete contractions of tensors of the form~\eqref{eqn:polydifferential-precontraction} with $\ell \geq j+1$;
 these will be considered next.
 Graham et al.\ observed~\cite{GJMS1992}*{Section~3} that $(\prod_{i\not=i_0}\cu_i)\cDelta^{n/2}\cu_{i_0}$ is independent of the ambiguity of $\cg$.
 Otherwise \cref{f-invariant} implies that the contribution of the factor $\cnabla^{M_{i_0}}\cu_{i_0}$ is independent of the ambiguity of $\cg$.
 If instead $i_0 \in \{ j+1, \dotsc, \ell\}$, then $\ell \geq j+1$ and hence Equation~\eqref{eqn:sum-strengths} yields
 \begin{equation*}
  S_{i_0} \leq \sum_{i=1}^\ell S_i - 4(\ell-j-1) = 2(k+2-\ell+j) \leq n+2 .
 \end{equation*}
 Moreover, equality holds if and only if $k=n/2$ and $\ell = j+1$ and $S_i = 0$ for  $i \not= i_0$.
 If equality holds, then our monomial is proportional to $\cu_0 \dotsm \cu_j \cI$ for $\cI$ a complete contraction of $\cnabla^{n/2}\cRm$.
 Otherwise $S_{i_0} \leq n+1$.
 In either case, Fefferman and Graham showed~\cite{FeffermanGraham2012}*{Propositions~6.2 and~9.1} that the contribution of the factor $\cnabla^{M_{i_0}}\cRm$ is independent of the ambiguity of $\cg$.
 We conclude that our monomial is independent of the ambiguity of $\cg$.
 
 Suppose now that $\ell \geq j+2$ and $k \leq n/2 + 1$.
 If $i_0 \in \{ 1, \dotsc, j \}$, then Equation~\eqref{eqn:sum-strengths} yields
 \begin{equation*}
  S_{i_0} \leq \sum_{i=1}^\ell S_i - 4(\ell-j) \leq 2(k-\ell+j) \leq 2k-4 \leq n-2 .
 \end{equation*}
 \Cref{f-invariant} then implies that the contribution of the factor $\cnabla^{M_{i_0}}\cu_{i_0}$ is independent of the ambiguity of $\cg$.
 If instead $i_0 \in \{ j+1, \dotsc, \ell \}$, then Equation~\eqref{eqn:sum-strengths} yields
 \begin{equation*}
  S_{i_0} \leq \sum_{i=1}^\ell S_i - 4(\ell-j-1) \leq 2(k+2-\ell+j) \leq 2k \leq n + 2 .
 \end{equation*}
 Moreover, equality holds if and only if $S_i = 0$ for each $i \leq j$ and $S_i=4$ for each $i \in \{ j+1, \dotsc, \ell \} \setminus \{ i_0 \}$.
 Fefferman and Graham showed~\cite{FeffermanGraham2012}*{Propositions~6.2 and~9.1} that the contribution of the factor $\cnabla^{M_{i_0}}\cRm$ is independent of the ambiguity of $\cg$.
 We again conclude that our monomial is independent of the ambiguity of $\cg$.
\end{proof}

Now let $D \colon \mE\bigl[-\frac{n-2k}{j+1}\bigr]^{\otimes j} \to \mE\bigl[-\frac{jn+2k}{j+1}\bigr]$ be a conformally invariant polydifferential operator.
Then for every compact conformal manifold $(M^n,\kc)$, the Dirichlet form $\kD \colon \mE\bigl[-\frac{n-2k}{j+1}\bigr]^{\otimes(j+1)} \to \bR$ determined by
\begin{equation*}
 \mathfrak{D}(u_0 \otimes \dotsm \otimes u_j) := \int_M u_0D(u_1 \otimes \dotsm \otimes u_j) \dvol ,
\end{equation*}
is conformally invariant.
We say that $D$ is \defn{formally self-adjoint} if $\kD$ is symmetric.
This implies that $D$ is itself symmetric.

We conclude this section by constructing the curved Ovsienko--Redou operators $D_{2k}$ and their linear analogues $D_{2k,\mI}$.
To that end, we identify a tangential linear combination of operators of the form $\cDelta^{k-j} \circ \cf \circ \cDelta^j$ for some fixed homogeneous function $\cf \in \cmE[w']$ (cf.\ \cite{CaseLinYuan2022or}*{Lemma~6.1}).

\begin{lemma}
 \label{construction-lemma}
 Let $(\cmG,\cg)$ be an ambient space for a conformal manifold $(M^n,\kc)$ and let $\cf \in \cmE[w']$.
 Let $k \in \bN$ and $w \in \bR$.
 Define $\cD_{2k,w,\cf} \colon \cmE[w] \to \cmE[w+w'-2k]$ by
 \begin{align*}
  \cD_{2k,w,\cf}(\cu) & := \sum_{j=0}^k a_j \cDelta^{k-j}\left( \cf \cDelta^j\cu \right) , \\
  a_{j} & := \binom{k}{j} \frac{\Gamma\bigl(j+k-w-w'-\frac{n}{2}\bigr)\Gamma\bigl(\frac{n}{2}+w-j\bigr)}{\Gamma\bigl(k-w-w'-\frac{n}{2}\bigr)\Gamma\bigl(\frac{n}{2}+w-k\bigr)} .
 \end{align*}
 Then $\iota^\ast\cD_{2k,w,\cf}(\cu)$ depends only on $\iota^\ast\cu$, $\cf$, and $\cg$.
\end{lemma}

\begin{proof}
 Consider the defining function $Q := \lv T \rv^2$ for $\iota(\mG) \subseteq \cmG$.
 It suffices to show that the commutator $[\cD_{2k,w,\cf},Q]=0$ on $\cmE[w-2]$, where $Q$ is regarded as a multiplication operator.
 Graham et al.\ observed~\cite{GJMS1992}*{Equation~(1.8)} that
 \begin{equation*}
  [ \cDelta^\ell , Q ] = 2\ell\cDelta^{\ell-1}( 2T + n + 4 - 2\ell)
 \end{equation*}
 on $C^\infty(\cmG)$ for any $\ell \in \bN$.
 Therefore
 \begin{equation*}
  [ \cD_{2k,w,\cf} , Q] = \sum_{j=0}^{k-1} b_j \cDelta^{k-j-1} \circ \cf \circ \cDelta^j
 \end{equation*}
 on $\cmE[w-2]$, where
 \begin{equation*}
  b_j = 2(j+1)(n+2w-2j-2)a_{j+1} - 2(k-j)( 2j+2k-2w-2w'-n )a_{j} .
 \end{equation*}
 Our definition of $a_j$ yields $b_j=0$, and hence $[\cD_{2k,w,\cf},Q]=0$ on $\cmE[w-2]$.
\end{proof}

Note that $\cD_{k,w,\cf}$ is not a differential operator in the sense of Subsection~\ref{subsec:ambient} because we do not require that $\cf$ is an ambient scalar Riemannian invariant.
Applying \cref{construction-lemma} with a suitably chosen ambient scalar Riemannian invariant does produce a tangential different operator:

\begin{corollary}
 \label{linear}
 Let $(M^n,\kc)$ be a conformal manifold and let $\cI \in \cmE[-2\ell]$ be an ambient scalar Riemannian invariant.
 Let $k \in \bN$;
 if $n$ is even, then assume additionally that
 \begin{enumerate}
  \item $k \leq n/2$, if $\ell=0$; and
  \item $k + \ell \leq n/2+1$, if $\ell \geq 1$.
 \end{enumerate}
 Then
 \begin{equation*}
 \cD_{2k,\cI}(\cu) := \sum_{r+s=k} \frac{k!}{r!s!}\frac{(\ell+s-1)!(\ell+r-1)!}{(\ell-1)!^2} \cDelta^r\left( \cI \cDelta^s\cu \right)
 \end{equation*}
 defines a tangential differential operator $\cD_{2k,\cI} \colon \cmE\bigl[-\frac{n-2k-2\ell}{2}\bigr] \to \cmE\bigl[-\frac{n+2k+2\ell}{2}\bigr]$.
 In particular, the differential operator
 \begin{align*}
  D_{2k,\mI} & \colon \mE\left[ -\frac{n-2k-2\ell}{2} \right] \to \mE\left[ -\frac{n+2k+2\ell}{2} \right] , \\
  D_{2k,\mI}(\iota^\ast\cu) & := (\cD_{2k,\cI}\cu) \circ \iota ,
 \end{align*}
 is conformally invariant.
\end{corollary}

\begin{proof}
 \Cref{dependence} implies that $\cD_{2k,\cI}$ is independent of the ambiguity of $\cg$.
 Since $\cI \in \cmE[-2\ell]$, applying \cref{construction-lemma} then implies that
 \begin{equation}
  \label{eqn:decompose-linear}
  \cD_{2k,\cI} = \cD_{2k,w,\cI}
 \end{equation}
 is tangential, where $w := -\frac{n}{2}+k+\ell$.
\end{proof}

The curved Ovsienko--Redou operators arise by looking for tangential linear combinations of the operators $\cD_{2k-2s,\cDelta^{s}\cf}$.

\begin{corollary}
 \label{ovsienko-redou}
 Let $(M^n,\kc)$ be a conformal manifold.
 Let $k \in \bN$;
 if $n$ is even, then assume additionally that $k \leq n/2$.
 Then
 \begin{align*}
  \cD_{2k}(\cu \otimes \cv) & := \sum_{r+s+t=k} a_{r,s,t} \cDelta^r\left( (\cDelta^s\cu)(\cDelta^t\cv) \right) , \\
 a_{r,s,t} & := \frac{k!}{r!s!t!}\frac{\Gamma\bigl(\frac{n+4k}{6}-r\bigr)\Gamma\bigl(\frac{n+4k}{6}-s\bigr)\Gamma\bigl(\frac{n+4k}{6}-t\bigr)}{\Gamma\bigl(\frac{n-2k}{6}\bigr)\Gamma\bigl(\frac{n+4k}{6}\bigr)^2},
 \end{align*}
 defines a tangential bidifferential operator $\cD_{2k} \colon \cmE\bigl[-\frac{n-2k}{3}\bigr]^{\otimes 2} \to \cmE\bigl[-\frac{2n+2k}{3}\bigr]$.
 In particular, the bidifferential operator
 \begin{align*}
  D_{2k} & \colon \mE\left[ -\frac{n-2k}{3} \right]^{\otimes 2} \to \mE\left[ -\frac{2n+2k}{3} \right] , \\
  D_{2k}( \iota^\ast\cu \otimes \iota^\ast\cv ) & := \cD_{2k}( \cu \otimes \cv ) \circ \iota
 \end{align*}
 is conformally invariant.
\end{corollary}

\begin{proof}
 \Cref{dependence} implies that $\cD_{2k}$ is independent of the ambiguity of $\cg$.
 Since $\cD_{2k}(\cu \otimes \cv) = \cD_{2k}(\cv \otimes \cu)$, it thus suffices to show that $\cD_{2k}(\cu \otimes \cv) \circ \iota$ depends only on $\iota^\ast\cv$, $\cu$, and $\cg$ in order to conclude that $\cD_{2k}$ is tangential.
 Observe that
 \begin{equation}
  \label{eqn:decompose-ovsienko-redou}
  \cD_{2k}(\cu \otimes \cv) = \sum_{s=0}^k \binom{k}{s}\frac{\Gamma\bigl(\frac{n+4k}{6}-s\bigr)\Gamma\bigl(\frac{n-2k}{6}+s\bigr)^2}{\Gamma\bigl(\frac{n-2k}{6}\bigr)\Gamma\bigl(\frac{n+4k}{6}\bigr)^2}\cD_{2k-2s,-\frac{n-2k}{3},\cDelta^s\cu}(\cv) .
 \end{equation}
 Applying \cref{construction-lemma} to each $\cD_{2k-2s,-\frac{n-2k}{3},\cDelta^s\cu}$ implies that $\cD_{2k}(\cu \otimes \cv) \circ \iota$ depends only on $\iota^\ast\cv$, $\cu$, and $\cg$.
\end{proof}

\section{Poincar\'e spaces}
\label{sec:scattering}

In this section we give an equivalent description of the operators $\cD_{2k,w,\cf}$ from \cref{construction-lemma} in terms of the Poincar\'e space determined by a straight and normal ambient metric.
We restrict our attention to the case $w = -\frac{n+w'-2k}{2}$, where $w'$ is the weight of $\cf$, as it is only in this case that the operator on $(M^n,g)$ induced by $\cD_{2k,w,\cf}$ can be formally self-adjoint.

Let $(\cmG,\cg)$ be the straight and normal ambient space~\eqref{eqn:straight-and-normal} for a representative $g \in \kc$ of a conformal manifold $(M^n,\kc)$.
Set $r := \sqrt{-2\rho}$ and $s := rt$ in the domain $\cmG_+ := \{ \rho < 0, t > 0 \} \subseteq \cmG$.
In these coordinates,
\begin{equation}
 \label{eqn:rewrite-expansion-of-g}
 \begin{split}
  \cg & = -ds^2 + s^2g_+ , \\
  g_+ & = r^{-2}(dr^2 + g_{-r^2/2}) .
 \end{split}
\end{equation}
The equation $\cg(T,T)=-1$ defines a hypersurface $\mathring{X} \subseteq \cmG$.
Since $T = t\partial_t$, we see that $\mathring{X} = \{ s=1 \}$.
Using Equation~\eqref{eqn:rewrite-expansion-of-g} to compute the Laplacian $\cDelta$ of $\cg$ in terms of the Laplacian of $g_+$ yields
\begin{equation}
 \label{eqn:poincare-laplacian}
 \cDelta = s^{-2}\left( \Delta_{g_+} - (s\partial_s)^2 - ns\partial_s \right) .
\end{equation}

Suppose that $\cu \in \cmE[w]$.
Thus there is a one-parameter family of functions $u_\rho \in C^\infty(M)$ such that $\cu(t,x,\rho) = t^wu_\rho(x)$.
Equivalently, $\cu(r,x,s) = s^w\hu(r,x)$, where $\hu \in C^\infty(\mathring{X})$ is the function $\hu(r,x) = r^{-w}u_{-r^2/2}(x)$.
In particular, $\cu$ is determined by its restriction $\hu = \cu\rv_{\mathring{X}}$.
Moreover, Equation~\eqref{eqn:poincare-laplacian} implies that
\begin{equation}
 \label{eqn:cDelta}
 \cDelta\cu = s^{w-2}\left( \Delta_{g_+} - w(n+w) \right)\hu .
\end{equation}
This leads to a formula relating the operators $\cD_{2k,w,\cf}$ and
\begin{equation*}
 P_{2j,w} := \prod_{i=0}^{j-1} \left( \Delta_{g_+} - (w-2i)(n+w-2i) \right)
\end{equation*}
on $(\mathring{X},g_+)$:

\begin{lemma}
 \label{expand-Df}
 Let $(\cmG,\cg)$ be the straight and normal ambient space for a representative $g \in \kc$ of a conformal manifold $(M^n,\kc)$.
 Let $(\mathring{X},g_+)$ be the Poincar\'e manifold determined by the variables $r,s$ as in Equation~\eqref{eqn:rewrite-expansion-of-g}.
 Let $\cf \in \cmE[w']$ and let $k \in \bN$.
 Set $w := -\frac{n+w'-2k}{2}$.
 If $\cu \in \cmE\bigl[ w \bigr]$, then
 \begin{align*}
  \cD_{2k,w,\cf}(\cu)\rv_{\mathring{X}} &= \sum_{i=0}^k\sum_{j=0}^{k-i} c_{i,j}( P_{2i,w} \circ \hf \circ P_{2j,w} )(\hu) ,\\
  c_{i,j}&:=4^{k-i-j}\frac{k!}{i!j!}\frac{\Gamma\bigl(\frac{n}{2}+w-i\bigr)\Gamma\bigl(\frac{n}{2}+w-j\bigr)}{\Gamma\bigl(\frac{n}{2}+w-k\bigr)^2} ,
 \end{align*}
where $\hu:=\cu|_{\mathring{X}}$ and $\hf:=\cf|_{\mathring{X}}$.
\end{lemma}

\begin{proof}
 Note that $w+w'-k=k-n-w$.
 \Cref{construction-lemma} and Equation~\eqref{eqn:cDelta} yield
 \begin{multline}
  \label{eqn:expand-Df1}
  \cD_{2k,w,\cf}(\cu)\rv_{\mathring{X}} = \sum_{j=0}^{k} \binom{k}{j}\frac{\Gamma\bigl(\frac{n}{2}+w-k+j\bigr)\Gamma\bigl(\frac{n}{2}+w-j\bigr)}{\Gamma\bigl(\frac{n}{2}+w-k\bigr)^2} \\
   \times (P_{2k-2j,2k-2j-n-w} \circ \hf \circ P_{2j,w}) (\hu) .
 \end{multline}
 On the one hand, a simple reindexing in the definition of $P_{2j,w}$ yields
 \begin{equation}
  \label{eqn:expand-Df2}
  P_{2k-2j,2k-2j-n-w} = P_{2k-2j,w-2} .
 \end{equation}
 On the other hand, direct computation yields
 \begin{equation*}
  \Delta_{g_+} - (w-2j)(n+w-2j) = \Delta_{g_+} - w(n+w) + 2j(n+2w-2j) .
 \end{equation*}
 It follows that
 \begin{equation*}
  P_{2j,w-2} = P_{2j,w} + 4j\left(\frac{n}{2}+w-j\right)P_{2j-2,w-2} .
 \end{equation*}
 A simple induction argument yields
 \begin{equation}
  \label{eqn:expand-Df3}
  P_{2j,w-2} = \sum_{i=0}^j 4^i\frac{j!}{(j-i)!}\frac{\Gamma\bigl(\frac{n}{2}+w-j+i\bigr)}{\Gamma\bigl(\frac{n}{2}+w-j\bigr)}P_{2j-2i,w} .
 \end{equation}
 Combining Equations~\eqref{eqn:expand-Df1}, \eqref{eqn:expand-Df2}, and~\eqref{eqn:expand-Df3} yields
 \begin{multline*}
  \cD_{2k,w,\cf}(\cu)\rv_{\mathring{X}} = \sum_{j=0}^k \sum_{i=0}^{k-j} 4^i\frac{k!}{j!(k-i-j)!}\frac{\Gamma\bigl(\frac{n}{2}+w-k+i+j\bigr)\Gamma\bigl(\frac{n}{2}+w-j\bigr)}{\Gamma\bigl(\frac{n}{2}+w-k\bigr)^2} \\
   \times (P_{2k-2i-2j,w} \circ \hf \circ P_{2j,w})(\hu) .
 \end{multline*}
 Reindexing the sum yields the final result.
\end{proof}

\section{Formal self-adjointness}
\label{sec:conclusion}

We are now ready to prove the formal self-adjointness of the curved Ovsienko--Redou operators $D_{2k}$ and the conformally invariant operators $D_{2k,\mI}$.
To that end, we identify their associated Dirichlet forms as the logarithmic term in the expansion of the corresponding Dirichlet form in a Poincar\'e space.
More generally:

\begin{lemma}
 \label{log-energy}
 Let $(\cmG,\cg)$ be the straight and normal ambient space for a representative $g \in \kc$ of a conformal manifold $(M^n,\kc)$.
 Denote by $(\mathring{X},g_+)$ the Poincar\'e space determined by the variables $r,s$ as in Equation~\eqref{eqn:rewrite-expansion-of-g}.
 Let $\cf \in \cmE[w']$ and let $k \in \bN$.
 Set $w := -\frac{n+w'-2k}{2}$.
 Define $D_{2k,w,\cf} \colon C^\infty(M) \to C^\infty(M)$ by
 \begin{equation*}
  D_{2k,w,\cf}(u) := \cD_{2k,w,\cf}(t^{w}\pi^\ast u)\rv_{t=1,\rho=0} .
 \end{equation*}
 If $\cu,\cv \in \cmE[w]$, then
 \begin{equation*}
   \int_{r>\varepsilon}  \left.\left(\cv \cD_{2k,w,\cf}(\cu) \right)\right|_{\mathring{X}} \dvol_{g_+} = -\left(\int_M v D_{2k,w,\cf}(u) \dvol_g \right)\log \varepsilon + O(1) 
 \end{equation*}
 as $\varepsilon \to 0$, where $u := \cu\rv_{t=1,\rho=0}$ and $v := \cv\rv_{t=1,\rho=0}$.
\end{lemma}

\begin{proof}
 Note that $2w+w'-2k=-n$.
 Thus $\cv \cD_{2k,w,\cf}(\cu) \in \cmE[-n]$.
 Therefore
 \begin{equation*}
  \bigl(\cv \cD_{2k,w,\cf}(\cu) \bigr)|_{\mathring{X}}=v D_{2k,w,\cf}(u) r^n+O(r^{n+2})
 \end{equation*}
 as $r \to 0$.
 Graham observed~\cite{Graham2000}*{Equations~(3.1) and~(3.2)} that
 \begin{equation*}
  \dvol_{g_+} = r^{-n-1} \left(1+O(r^2)\right) \dvol_{g} dr 
 \end{equation*}
 as $r \to 0$.
 Direct calculation yields the final result.
\end{proof}

Combining \cref{expand-Df,log-energy} implies that $D_{2k,w,\cf}$ is formally self-adjoint.

\begin{proposition}
 \label{fsa-key-step}
 Let $(\cmG,\cg)$ be the straight and normal ambient space for a representative $g \in \kc$ of a conformal manifold $(M^n,\kc)$.
 Denote by $(\mathring{X},g_+)$ the Poincar\'e space as in Equation~\eqref{eqn:rewrite-expansion-of-g}.
 Let $\cf \in \cmE[w']$ and let $k \in \bN$.
 Set $w := -\frac{n+w'-2k}{2}$.
 Define $D_{2k,w,\cf} \colon C^\infty(M) \to C^\infty(M)$ as in \cref{log-energy}.
 Then
 \begin{equation*}
  \int_M v D_{2k,w,\cf}u \dvol_g = \int_M u D_{2k,w,\cf}(v) \dvol_g
 \end{equation*}
 for every compactly supported $u,v \in C^\infty(M)$.
\end{proposition}

\begin{proof}
 Let $u,v \in C^\infty(M)$ be compactly supported.
 Set $\cu := t^w\pi^\ast u$ and $\cv := t^w\pi^\ast v$.
 Given a function $\phi = \phi(\varepsilon)$ such that $\phi \in O(\log\varepsilon)$ as $\varepsilon\to 0$, denote
 \begin{equation*}
  \logpart \phi := \lim_{\varepsilon\to0} \frac{\phi(\varepsilon)}{\log(1/\varepsilon)} .
 \end{equation*}
 First, \cref{log-energy} yields
 \begin{multline}
  \label{eqn:compute-adjoint}
  \int_M \left( vD_{2k,w,\cf}(u) - uD_{2k,w,\cf}(v) \right) \dvol_g \\
   = \logpart \int_{r > \varepsilon} \left.\left( \cv \cD_{2k,w,\cf}(\cu) - \cu \cD_{2k,w,\cf}(\cv) \right)\right|_{\mathring{X}} \dvol_{g_+} .
 \end{multline}
 Second, \cref{expand-Df} and the symmetry $c_{i,j} = c_{j,i}$ yield
 \begin{multline}
  \label{eqn:apply-expand-Df}
  \int_{r > \varepsilon} \left.\left( \cv \cD_{2k,w,\cf}(\cu) - \cu \cD_{2k,w,\cf}(\cv) \right)\right|_{\mathring{X}} \dvol_{g_+} \\
   = \sum_{i+j=0}^{k} c_{i,j} \int_{r>\varepsilon} \left( \hv (P_{2i,w} \circ \hf \circ P_{2j,w})(\hu) - \hu (P_{2j,w} \circ \hf \circ P_{2i,w})(\hv) \right) \dvol_{g_+} ,
 \end{multline}
 where $\hu := \cu \rv_{\mathring{X}}$ and $\hv := \cv \rv_{\mathring{X}}$.
 Third, observe that if $0 \leq j \leq i$, then
 \begin{equation*}
  P_{2i,w} = P_{2i-2j,w-2j} \circ P_{2j,w} .
 \end{equation*}
 Let $\hvarphi,\hpsi \in C^\infty(\mathring{X})$.
 Write
 \begin{align*}
  \MoveEqLeft[1] \int_{r>\varepsilon} \left( \hvarphi P_{2i,w}\hpsi - \hpsi P_{2i,w}\hvarphi \right) \dvol_{g_+} \\
   & = \sum_{\ell=1}^{i} \int_{r>\varepsilon} \left( (P_{2i-2\ell,w-2\ell}\hvarphi)P_{2\ell,w}\hpsi - (P_{2i-2\ell+2,w-2\ell+2}\hvarphi)P_{2\ell-2,w}\hpsi \right) \dvol_{g_+} \\
   & = \sum_{\ell=1}^{i} \int_{r>\varepsilon} \left( (P_{2i-2\ell,w-2\ell}\hvarphi)\Delta_{g_+}P_{2\ell-2,w}\hpsi - (\Delta_{g_+}P_{2i-2\ell,w-2\ell}\hvarphi)P_{2\ell-2,w}\hpsi \right) \dvol_{g_+} .
 \end{align*}
 The Divergence Theorem yields
 \begin{align*}
  \MoveEqLeft[1] \int_{r>\varepsilon} \left( \hvarphi P_{2i,w}\hpsi - \hpsi P_{2i,w}\hvarphi \right) \dvol_{g_+} \\
   & = \sum_{\ell=1}^{i} \int_{r=\varepsilon} \left( (P_{2i-2\ell,w-2\ell}\hvarphi)\nabla_{\onf}P_{2\ell-2,w}\hpsi - (\nabla_{\onf}P_{2i-2\ell,w-2\ell}\hvarphi)P_{2\ell-2,w}\hpsi \right) \darea_{g_+} ,
 \end{align*}
 where $\onf$ and $\darea_{g_+}$ are the outward-pointing unit normal and the area element, respectively, of $\{ r=\varepsilon \}$ with respect to $g_+$.
 In particular,
 \begin{equation*}
  \logpart \int_{r>\varepsilon} \left( \hvarphi P_{2i,w}\hpsi - \hpsi P_{2i,w}\hvarphi \right) \dvol_{g_+} = 0 .
 \end{equation*}
 Combining this with Equation~\eqref{eqn:compute-adjoint} and~\eqref{eqn:apply-expand-Df} yields the final result.
\end{proof}

Our main results follow from \cref{fsa-key-step} and tangentiality.

\begin{proof}[Proof of \cref{ovsienko-redou-are-fsa}]
 By \cref{ovsienko-redou}, it suffices to compute in the straight and normal ambient space associated to a representative $g \in \kc$.
 Equation~\eqref{eqn:decompose-ovsienko-redou} and \cref{fsa-key-step} express $D_{2k}$ as a linear combination of formally self-adjoint operators.
 Therefore $D_{2k}$ is formally self-adjoint.
\end{proof}
    
\begin{proof}[Proof of \cref{linear-are-fsa}]
 By \cref{linear}, it suffices to compute in the straight and normal ambient space associated to a representative $g \in \kc$.
 Combining Equation~\eqref{eqn:decompose-linear} with \cref{fsa-key-step} realizes $D_{2k,\mI}$ as a formally self-adjoint operator.
\end{proof}

\section*{Acknowledgements}
We would like to thank Shane Chern for very useful discussions.

JSC acknowledges support from a Simons Foundation Collaboration Grant for Mathematicians, ID 524601.

\bibliography{bib}

\end{document}